\documentclass[12pt,a4paper,final]{article}

\usepackage[margin=1in]{geometry}

\usepackage{setspace} 

\usepackage{amscd}
\usepackage{amsfonts}
\usepackage{amsmath}
\usepackage{amssymb}
\usepackage{amstext}
\usepackage{amsthm}
\usepackage{color}
\usepackage[nodate]{datetime}
\usepackage{dsfont}
\usepackage{enumerate}
\usepackage{fancyhdr}
\usepackage{framed}
\usepackage{graphicx}
\usepackage[hidelinks]{hyperref}
\usepackage{mathrsfs}
\usepackage{pifont}
\usepackage{psfrag}
\usepackage[normalem]{ulem}
\usepackage{url}
\usepackage{verbatim}
\usepackage{wasysym}

\usepackage[utf8]{inputenx}
\usepackage[T1]{fontenc}
\usepackage{lmodern}

\hypersetup{colorlinks=true,urlcolor=blue,linkcolor=black,citecolor=black}
\newcommand*{\doi}[1]{\href{http://dx.doi.org/\detokenize{#1}}{doi}}

\newcommand{\E}{\mathbb{E}}
\newcommand{\EE}{\mathbb{E}}

\renewcommand{\H}{\mathcal{H}}
\newcommand{\I}{\mathds{1}}

\newcommand{\N}{\mathbb{N}}
\newcommand{\NN}{\mathbb{N}}
\renewcommand{\o}{\mathbf{0}}
\newcommand{\PPP}{\mathcal{P}}
\newcommand{\PP}{\mathbb{P}}
\newcommand{\Q}{\mathbb{Q}}
\newcommand{\R}{\mathbb{R}}

\renewcommand{\v}{{\boldsymbol{v}}}

\newcommand{\X}{\mathbf{X}}
\newcommand{\Y}{\mathbf{Y}}
\newcommand{\Z}{\mathbb{Z}}
\newcommand{\ZZ}{\mathbb{Z}}

\newcommand{\limites}[2]{\overset{#1}{\underset{#2}{\longrightarrow}}}

\theoremstyle{plain}
\newtheorem{theorem}{Theorem}

\newtheorem{lemma}[theorem]{Lemma}
\newtheorem{proposition}[theorem]{Proposition}

\newtheorem*{assumption*}{Assumption}
\newtheorem*{lemma*}{Lemma}
\newtheorem*{proposition*}{Proposition}
\newtheorem*{theorem*}{Theorem}
\theoremstyle{definition}
\newtheorem*{definition*}{Definition}
\theoremstyle{remark}

\newtheorem*{remark*}{Remark}

\date{}

\renewcommand{\geq}{\geqslant}
\renewcommand{\ge}{\geqslant}
\renewcommand{\leq}{\leqslant}
\renewcommand{\le}{\leqslant}
\renewcommand{\subset}{\subseteq}
\renewcommand{\supset}{\supseteq}

\renewcommand{\baselinestretch}{1.15} 
\definecolor{orange}{cmyk}{0,0,0.2,0}
\definecolor{gray}{rgb}{.5,.5,.5}

\newcommand{\Vt}{\widetilde V}
\newcommand{\Mt}{\widetilde M}

\newcommand{\lest}{\le_{\rm st}} 

\newcommand{\Ir}{\mathcal I}
\newcommand{\pv}{\,;\,} 
\newcommand{\pp}{\,:\,} 

\newcommand{\Pb}{\boldsymbol{\mathcal P}}

\newcommand{\zetab}{\boldsymbol{\zeta}}
\newcommand{\xib}{\boldsymbol{\xi}}

\newcommand{\etal}{\overline\eta}
\newcommand{\etalb}{\boldsymbol{\overline\eta}}

\newcommand{\influence}[1]{\overset{}{\underset{#1}{\leadsto}}}

\newcommand{\Zt}{\widetilde Z}

\begin{document}

\title{Non-fixation for Biased Activated Random Walks}

\author{
L. T. Rolla\footnote{Argentinian National Research Council at the University of Buenos Aires, NYU-ECNU Institute of Mathematical Sciences at NYU Shanghai. \emph{E-mail address:} \href{mailto:leorolla@dm.uba.ar}{\nolinkurl{leorolla@dm.uba.ar}}} \qquad
L. Tournier\footnote{Université Paris 13, Sorbonne Paris Cité, LAGA, CNRS UMR 7539, F-93430 Villetaneuse, France. \emph{E-mail address:} \href{mailto:tournier@math.univ-paris13.fr}{\nolinkurl{tournier@math.univ-paris13.fr}}}
}

\maketitle

\begin{abstract}
We prove that the model of Activated Random Walks on $\ZZ^d$ with biased jump distribution does not fixate for any positive density, if the sleep rate is small enough, as well as for any finite sleep rate, if the density is close enough to $1$.
The proof uses a new criterion for non-fixation. 
We provide a pathwise construction of the process, of independent interest, used in the proof of this non-fixation criterion.
\end{abstract}


This preprint has the same numbering of sections, equations and theorems as the the
published article
``\emph{Ann. Inst. H. Poincaré Probab. Statist. 54 (2018): 938--951.}''



\section{Introduction}

The Activated Random Walk model is a system of particles in either one of two states, active or passive. Active particles perform continuous-time independent random walks on $\Z^d$ with jump distribution $p(\cdot)$, and switch to passive state at rate $\lambda>0$ when they are alone on a site. Passive particles do not move but are reactivated immediately when visited by another particle. Initially, all particles are active.
In this process, competition occurs between the global spread of activity by diffusion and spontaneous local deactivations, which leads to a transition between two regimes, as either $\lambda$ decreases or the initial density $\mu$ of particles increases. For small $\mu$ or large $\lambda$, the system is expected to fixate, i.e., the process in any finite box eventually remains constant. For small $\lambda$ or $\mu$ close to~$1$, activity is believed to persist forever.

Fixation for small $\mu$ or large $\lambda$ was proved by Rolla and Sidoravicius~\cite{RollaSidoravicius12} in dimension~$1$.
Non-fixation was proved by Shellef~\cite{Shellef10} and by Amir and Gurel-Gurevich~\cite{AmirGurel-Gurevich10} when $\mu > 1$, and by Cabezas, Rolla and Sidoravicius~\cite{CabezasRollaSidoravicius14} when $\mu=1$.

These results have been recently extended or sharpened, under restrictive assumptions on the walks. In the case when particles perform simple symmetric random walks, fixation was proved in any dimension by Sidoravicius and Teixeira~\cite{SidoraviciusTeixeira17} for sufficiently small $\mu$. Basu, Ganguly and Hoffman~\cite{BasuGangulyHoffman15} showed that, in dimension $1$, non-fixation happens for sufficiently low sleep rate.

In the case of biased random walks, Taggi~\cite{Taggi16} proved non-fixation if the density $\mu$ is sufficiently close to~$1$.
More precisely, under the assumption that the jump distribution is biased, Taggi proved that, for $d=1$, non-fixation happens if $\mu > 1-F(\lambda)$, where $F(\lambda)$ is positive and tends to~$1$ as $\lambda\to0$ (see below for the definition of $F$).
In higher dimensions, Taggi shows that non-fixation happens when $\mu F(\lambda) > \nu_0$, where $\nu_0$ is the density of initially unoccupied sites.
As a consequence, for some particular families of distributions, non-fixation happens for some pair of $\lambda>0$ and $\mu<1$.
It also implies non-fixation for arbitrary $\lambda$ in case $\nu_0$ gets small as $\mu\to 1$ (e.g.\ for Bernoulli initial condition).

In this paper we extend the latter result, proving that, for biased walks in any dimension, non-fixation happens if $\mu > 1-F(\lambda)$.
Thus, for arbitrarily low density $\mu$, non-fixation happens provided the sleep rate $\lambda$ is small enough, and, for arbitrarily high $\lambda$, it happens provided $\mu$ is close enough to~$1$.
Let us give the precise statement.

For any vector $\v\in\R^d \setminus \{\o\}$, define the half-space
\(\H_\v=\{x\in\ZZ^d\,:\,x\cdot \v\le 0\}\)
and let
\begin{equation}\label{eq:defF}
F_\v(\lambda)=\E\big[ ({1+\lambda})^{-\ell_{\H_\v}}\big]
\end{equation}
where $\ell_{\H_\v}$ is the total time spent in $\H_\v$ by a discrete-time random walk with jump distribution $p(\cdot)$.
Notice that, if $X_n\cdot \v\to+\infty$ a.s., then $F_\v(\lambda) >0 $ for any $\lambda$, and $F_\v(\lambda) \to 1$ as $\lambda\to0^+$.

\begin{theorem}\label{thm:main}
For the Activated Random Walk on $\ZZ^d$ with sleep rate $\lambda$ and i.i.d.~initial condition with mean $\mu$,
if $\mu > 1 - F_\v(\lambda)$ for some $\v$, then a.s.\ the system does not fixate.
\end{theorem}

Existence of a stochastic process corresponding to the ARW follows from general results about particle systems on non-compact state spaces~\cite{Andjel82}.
Equivalence between fixation of such process and stability of the Diaconis-Fulton site-wise representation was established in~\cite{RollaSidoravicius12}, yielding a framework suitable for the study of the fixation vs.\ non-fixation question (see Section~\ref{sec:sitewise} for a description of the site-wise representation and Section~\ref{sec:globalsitewise} for its relationship with the continuous-time particle system).
The abelian property of the site-wise representation implies that one can choose the order at which particles will move, and a \emph{stabilizing strategy} refers to a particular procedure to determine such choice.

The proof of~\cite{Taggi16} for $d=1$ consists in a stabilizing strategy where one moves the particles in a very coordinated fashion, and uses~(\ref{eq:defF}) to bound the number of particles lost on the way. It is based on a criterion from~\cite{RollaSidoravicius12}, which in particular reduces the question of non-fixation to the following condition.
For $n\in\N$, let $\PP_n$ denote the law of the ARW process whose evolution is restricted to $V_n=\{-n,\ldots,n\}^d$, that is, particles freeze outside $V_n$.
For $n\in\N$, fix $x_n\in V_n$ and denote by $K_n$ the number of visits to $x_n$.

\begin{proposition}
\label{pro:oldcondition}
If $\PP_n(K_n \geq C) \to 1$ for any $C$, then a.s.\ the system does not fixate.
\end{proposition}

In order to use a similar strategy in $d\ge 2$, \cite{Taggi16} considers ghost particles, which provide some independence. This is needed as the above criterion requires a control in probability on the number of visits to a given site.

The proof given here is a simplified version of Taggi's argument for $d=1$, and yet it works on any dimension.
This is made possible by the use of a new criterion for non-fixation.
Let $M_n$ count the number of particles that ever exit $V_n$.

\begin{proposition}
\label{pro:condition}
If
\(
\limsup_n \dfrac{\E_n M_n}{|V_n|}>0,
\)
then a.s.\ the system does not fixate.
\end{proposition}

Although the above criterion gives a condition that can be verified in terms of the site-wise representation, the proof uses the notion of \emph{particle fixation}.
We say that a given particle \emph{fixates} if it stays in passive state forever after a random finite time.
This definition makes sense for a process where each particle carries a label, so that it can be followed through time. We shall refer to this as the \emph{labeled version} of the process.
In contrast, the usual \emph{unlabeled version} of a particle system only considers the number of particles at each site without distinguishing among them.

One way to study the labeled system is to construct it explicitly.
In the \emph{particle-wise construction}, each particle carries a random path and a Poisson clock.
The evolution of such particle is given by a time-change of its own path, depending both on its Poisson clock and on the interaction with other particles.
This approach can be extremely useful for the use of arguments based on ergodicity, mass-transport principle, resampling, etc.
The following was proved in~\cite{AmirGurel-Gurevich10} using these tools.

\begin{theorem}
\label{thm:agg}
Assume that the particle-wise construction is almost surely well-defined.
Suppose that the initial condition is i.i.d.
If a particle is non-fixating with positive probability, then a.s.\ the system does not fixate.
\end{theorem}

To prove Proposition~\ref{pro:condition}, we use a non-abelian variation of the site-wise construction to obtain a sequence of finite approximations to the labeled process and show that \( \limsup_n \frac{\E_n M_n}{|V_n|}>0 \) implies positive probability for a labeled particle not to fixate.
Now in order to apply Theorem~\ref{thm:agg}, we need to prove that the labeled version of the ARW model can be constructed from this collection of random paths and Poisson clocks in a translation-covariant way, via an argument of almost-sure convergence.
This is the content of the next result, whose precise meaning is given in Section~\ref{sec:particlewise}.

\begin{theorem}\label{thm:well_defined_intro}
The particle-wise construction is almost surely well-defined.
\end{theorem}

Statements and proofs are given in $\Z^d$ for simplicity.
The proof of Theorem~\ref{thm:well_defined_intro} works in the setting of transitive unimodular graphs, and the proof of Proposition~\ref{pro:condition} works on the same setting with the further assumption of amenability.
See Section~\ref{sec:extensions} for details.

The next section exposes the site-wise representation and the local site-wise construction.
Assuming the non-fixation criterion, i.e.~Proposition~\ref{pro:condition}, this paves the way to proving Theorem~\ref{thm:main} in Section~\ref{sec:proof_activity}. In the aim of proving this criterion in Section~\ref{sec:nonfixation}, we first expose the formal definition of the particle-wise construction in Section~\ref{sec:particlewise} and state its well-definedness, whose proof is deferred to Section~\ref{sec:proof_existence}. Finally, Section~\ref{sec:globalsitewise} presents the global site-wise construction and Section~\ref{sec:extensions} discusses the extensions of the results to more general graphs.

\section{Notation and site-wise representation}
\label{sec:sitewise}

In this section we briefly describe the site-wise representation and state its abelian property.
This property will be used when proving Theorem~\ref{thm:main} with the aid of Proposition~\ref{pro:condition}.
Properties used in the proof of Proposition~\ref{pro:condition} itself will be discussed during the proof.
Denote $\N=\{1,2,3,\ldots\}$ and $\N_0=\{0,1,2,3,\ldots\}$, and $o$ as the origin in $\Z^d$.

Let $\eta\in(\N_{0\varrho})^{\Z^d}$ denote a configuration of particles.
Here
\(\N_{0\varrho}=\{0,\varrho,1,2,3,\ldots\},\)
and the symbol $\varrho$ corresponds to the presence of one passive particle (there may be at most one), while $1$ represents one active particle.

Let $\Ir^x=(\Ir^{x,k})_{k\in\N}$ denote the sequence of instructions assigned to each site $x$, where $\Ir^{x,k}\in\Z^d\cup\{\varrho\}$.
The \emph{toppling operation} consists in performing the action indicated by the first unused instruction $\iota$ in the sequence $\Ir^x$ assigned to $x$, which is to make a particle present at that site $x$ jump to $x+y$ (and activate a possibly present passive particle there) when $\iota=y\in\Z^d$, or try to fall asleep (succeeding only if the particle is alone) when $\iota=\varrho$.
In other words, the system goes through the transition $\eta\to\tau^{x,\iota}\eta$ where
\[
\tau^{x,y}\eta=\eta-\delta_x+\delta_{x+y},\qquad\text{and}\qquad \tau^{x,\varrho}\eta=\begin{cases}\eta-\delta_x+\varrho\delta_x&\text{if $\eta(x)=1$,}\\
\eta&\text{otherwise,}\end{cases}
\]
with the convention that $\varrho+1=2$ and $1-1+\varrho=\varrho$. 
In the sequel we discuss some combinatorial properties of this operation, in a deterministic setting.

\subsection{Topplings and stability}

Assume that $(\Ir^{x,k})_{k\in\N,x\in\Z^d}$ is fixed.
Let $\eta\in(\N_{0\varrho})^{\Z^d}$ and $h\in(\N_0)^{\Z^d}$. Here, $\eta$ can be thought of as a configuration reached by an ARW and $h$ as the number of instructions that have already been executed at each site to get to $\eta$.

A site~$x$ is \emph{unstable} in the configuration~$\eta$ if~$x$ contains active particles in that configuration, and \emph{stable} otherwise.
The \emph{toppling} of an unstable site $x$ is an operation on $\eta$ and $h$, denoted by $\Phi_x$, which changes the configuration $\eta$ according to the next unused instruction at $x$, and increases the counter $h$ by~$1$ at $x$: 
\[
\Phi_x(\eta,h)=(\tau^{x,\Ir^{x,h(x)+1}}\eta,h+\delta_x)
.
\]
We write $\Phi_x\eta$ for $\Phi_x(\eta,0)$. 
Toppling a site is \emph{legal} for $\eta$ if the site is unstable in $\eta$.

Let $\alpha=(x_1,\dots,x_n)$ denote a finite sequence of sites in~$\Z^d$, which we think of as the order in which a sequence of topplings will be applied.
The sequence $\alpha$ is said to be \emph{legal} for $\eta$ if each subsequent toppling involved in $\Phi_{x_n}(\cdots(\Phi_{x_1}\eta))$ is legal, in which case $\Phi_\alpha\eta$ can be defined, where  $\Phi_\alpha=\Phi_{x_n}\cdots\Phi_{x_1}$.
Let~$V$ denote a finite subset of~$\Z^d$. A configuration $\eta$ is said to be \emph{stable in~$V$} if all the sites $x\in V$ are stable in $\eta$.
We write {$\alpha \subset V$} if all elements of $\alpha$ are in~$V$.
We say that \emph{$\alpha$ stabilizes $\eta$ in $V$} if the configuration in $\Phi_\alpha \eta$ is stable in $V$.
Let $m_\alpha \in \N_0^{\Z^d}$ count the number of times each site appears in $\alpha$.

\begin{lemma*}[Local Abelianness]\label{lem:local_abelian}
If $\alpha$ and $\beta$ are legal sequences of topplings for $\eta$ such that $m_\alpha = m_\beta$, then $\Phi_\alpha\eta = \Phi_\beta\eta$.
\end{lemma*}

\begin{lemma*}[Global Abelianness]\label{lem:global_abelian}
If $\alpha$ and $\beta$ are both legal toppling sequences for $\eta$ that are contained in $V$ and stabilize $\eta$ in $V$, then $m_\alpha=m_\beta$.
In particular, $\Phi_\alpha\eta=\Phi_\beta\eta$.
\end{lemma*}

For the proofs, see e.g.~\cite{RollaSidoravicius12}.

\subsection{Site-wise construction of finite systems}
\label{sec:sitewisefinite}

We now describe the site-wise construction of a continuous-time stochastic process $(\eta_t)_{t \ge 0}$ corresponding to a system that contains finitely many particles.
The variable $\eta_t(x) \in \N_{0\varrho}$ represents the number and type of particles at the site $x\in\Z^d$ at time $t\ge0$.

Let the sleep rate $\lambda>0$ and the jump distribution $p(\cdot)$ on $\Z^d$ with $p(o)\neq1$ be given.
For each site $x$, sample a càdlàg process $t \mapsto N^x(t)\in \N_0$ that starts at $N^x(0)=0$ and jumps by $+1$ at rate given by $1+\lambda$.
Sample an i.i.d.\ sequence $\Ir^x$ with
\[
\Ir^{x,k}=
\bigg\{
\begin{array}{cl}
y & \text{with probability $\frac1{1+\lambda}p(y)$},\\
\varrho & \text{with probability $\frac\lambda{1+\lambda}$.}
\end{array}
\]
Sample the initial condition $\eta_0$ according to a prescribed probability distribution.

To define the continuous-time evolution $(\eta_t)_{t \ge 0}$, which will be piecewise constant and right continuous, let us first associate to it a local time process $(L^x_t)_{x\in\Z^d,\,t\ge0}$.
For each $x\in\Z^d$, $(L^x_t)_{t\ge0}$ starts at $0$, it is continuous and piecewise linear with slope given by the number $\eta_t(x)\I_{\{\eta_t(x)\ge1\}}$ of active particles at $x$.
For each $t$ and $x$, we define $h_t(x)=N^x(L^x_t)$.
Note that $h_0=0$ and, unless all particles are passive, there are some sites $x$ for which $L^x$ is increasing at time $0^+$. The process $\eta$ is kept constant equal to $\eta_0$ until the first time $t_1$ when $h$ jumps, which has to be to $\delta_{x_1}$ for some $x_1$. At this time, we let $\eta$ jump to $\eta_{t_1}=\tau^{\Ir^{x_1,1}}\eta_{0}$. The process $\eta$ is then kept constant until the next jump of $h$, which occurs at a time $t_2$ and to $h_{t_1}+\delta_{x_2}$ for some $x_2$. We let at this time $\eta$ jump to $\eta_{t_2}=\tau^{\Ir^{x_2,h_t(x_2)}}\eta_{t_1}$, and so on. 
This procedure can be continued until a stable configuration is reached.
Notice that we have $(\eta_t,h_t)=\Phi_{x_1,\dots,x_k}(\eta_0,h_0)$ for $t\in[t_k,t_{k+1})$.

Given a probability distribution on $(\N_{0\varrho})^{\Z^d}$ and finite $V \subseteq \Z^d$, let $\PP_{[V]}$ denote the law of $(\eta_0,\Ir,N)$ if $\Ir$ and $N$ are sampled as above, and $\eta_0$ is first sampled according to such distribution and then $\eta_0(x)$ is replaced by $-\infty$ for $x \not\in V$.
In words, $\PP_{[V]}$ corresponds to a dynamics where particles are erased as soon as they leave $V$.
For $\PP_n$ and $V_n$ defined in the Introduction, we have $\PP_n = \PP_{[V_n]}$.

By the abelian property, the stable configuration eventually reached by $(\eta_t)_{t\geq 0}$ depends on $\eta_0$ and $\Ir$, not on the Poisson clocks $N^x$.
This will be crucial in the next section.

\section{Non-fixation for biased ARW}
\label{sec:proof_activity}

In this section we prove Theorem~\ref{thm:main}.
Let us start with a preliminary observation.
Consider a discrete-time random walk with jump distribution $p$ starting at $o$.
Let a particle follow this random walk, but for each time the random walk is about to jump from a site in $\H_\v$ (there are $\ell_{\H_\v}$ such times), toss a coin and destroy the particle with probability $\frac{1}{1+\lambda}$.
The number $F_\v(\lambda)$ defined in~(\ref{eq:defF}) is the probability that the particle survives forever.

Assume $\mu>1-F_\v(\lambda)$.
By Proposition~\ref{pro:condition}, we have to show that
\begin{equation}
\label{eq:condition}
\limsup_n \frac{\E_{n} M_n}{|V_n|}>0.
\end{equation}
By the discussion in \S\ref{sec:sitewisefinite}, we may use the explicit site-wise construction of $\PP_n$ and in particular its abelian property.
The proof below consists in describing a toppling strategy that allows to show that, on average, sufficiently many particles (a positive density of them) quit $V_n$.

Let us choose a labeling $V_n=\{x_1,\ldots,x_r\}$, where $r=|V_n|$ and
\(x_1\cdot\v\le\cdots\le x_r\cdot\v\).
For $i=0,\ldots,r$, let \(A_i=\{x_{i+1},\ldots,x_r\}\subset V_n.\)

\textbf{Stage 1: leveling.}
We first make legal topplings inside $V_n$ until reaching a \emph{leveled configuration}, that is, a configuration where each site in $V_n$ contains at most one particle, and all particles inside $V_n$ are active.
This can be achieved by repeating the following procedure: locate the sites in $V_n$ containing at least two particles, and topple each of them once.
Since the number of particles in $V_n$ cannot increase during this procedure, a finite number of states can be reached from the initial one.
Therefore, repeating this procedure leads a.s.\ to a leveled state, as required.

\textbf{Stage 2: rolling.}
This stage subdivides into $r$ successive steps.
For $i=1,\ldots,r$, Step~$i$ goes as follows: if there is a particle at $x_i$ then we do a sequence of legal topplings, starting by a toppling at $x_i$, then at the location where this toppling pushed the particle to, and so on until this particle either (i) exits $V_n$, (ii) reaches an empty site in $A_{i}$, or (iii) falls asleep in $V_n$.
If the latter happens, we declare that the particle is \emph{left behind}.
Because of the stopping condition~(ii), condition~(iii) can only happen at $V_n\setminus A_{i}\subset x_i+\H_\v$.

By induction we have that, for $i=1,\ldots,r$, after Step $i$, the configuration is leveled on $A_i$.
In particular, at the start of Step $i$, there is at most one particle at $x_i$.
Therefore, at most one particle is left behind at each step, which happens with probability bounded from above by $1-F_\v(\lambda)$.

\textbf{Stage 3: finishing.}
Some particles in the previous stage may have activated particles previously left behind, so the final state may not be stable.
In this last stage we just stabilize the configuration inside $V_n$ by whatever sequence of topplings.

\textbf{Density of quitters.}
During Stage~2, each particle either ends up exiting $V_n$ or being left behind (if a given step ends due to condition (ii), the corresponding particle will be moved again in a later step).
Let $N_n$ denote the number of particles left behind during Stage~2.
Notice that $M_n \ge \sum_{x\in V_n} \eta_0(x) - N_n$.
Therefore,
$\E_n[M_n] \ge |V_n|\, \mu - |V_n|\,(1-F_\v(\lambda))$, completing the proof.

\section{Particle-wise construction}
\label{sec:particlewise}

Consider a representation of the ARW process where each particle is individually labeled, so that it can be followed through time, similarly to the informal description at the Introduction.
Such a labeled process determines the unlabeled evolution $(\eta_t)_{t \ge 0}$, by simply counting particles present at each site, but also contains significantly more information.

Given an initial condition $\eta_0$, each particle is assigned a label $(x,i)$ where $x\in \Z^d$ is the initial location and $i\in\{1,\ldots,\eta_0(x)\}$ distinguishes between particles starting at the same site.
The randomness will not be attached to the sites as in the previous sections, but to the particles instead.

\paragraph{Particle-wise randomness.}

Assign to each particle $(x,i)\in \Z^d\times\N$ a continuous-time walk $X^{x,i}=(X^{x,i}_t)_{t \geqslant 0}$ with jump rate 1 and jump distribution $p(\cdot)$, independently of anything else, as well as a Poisson clock $\PPP^{x,i}\subset \R_+$ according to which the particle will try to sleep.
$X^{x,i}$ will be the path of the particle parameterized by its \emph{inner time}, which is meant to halt when the particle becomes deactivated, and to resume at reactivation.
For this reason, $X^{x,i}$ will be called the \emph{putative trajectory} of particle $(x,i)$.
These elements, together with the initial configuration $\eta_0$, which can also be random, will be denoted by $\xib=(\eta_0,\X,\Pb)$.

\paragraph{Case of finite configurations.}

Given $\xib=(\eta_0,\X,\Pb)$ where $\eta_0$ contains finitely many particles (i.e., $\sum_x\eta_0(x)<\infty$), the particle-wise ARW is a simple continuous-time Markov chain on a countable state-space, and it is straightforward to define it explicitly using the trajectories $\X$ and clocks $\Pb$.
Let us denote this process by
\[(\zeta^{x,i}_t\pv x\in \Z^d,\,i\in\N)_{t\ge0}=(Y^{x,i}_t,\gamma^{x,i}_t\pv x\in \Z^d,\,i\in\N)_{t\ge0},\]
where $Y^{x,i}_t\in \Z^d\cup\{\Upsilon\}$ stands for the position of particle $(x,i)$ if $i\le \eta_0$ and $Y^{x,i}=\Upsilon$ if $i>\eta_0(x)$, and $\gamma^{x,i}_t\in\{1,\varrho\}$ stands for its state.

The pair $\zetab=(\Y,\boldsymbol{\gamma})$ describes the whole evolution of the system, and can be seen as a function of $\xib$ so we can write $\zetab=\zetab(\xib)=\zetab(\eta_0,\X,\Pb)$.

\paragraph{Extension to the infinite system.} We are now interested in defining $\zetab(\eta_0,\X,\Pb)$ when $\eta_0$ contains infinitely many particles.

Let us introduce a notation for the ``fully-labeled occupation measure'' of $\zetab$.
Assuming that $\eta_0$ has finitely many particles, for any site $z$ and time $t$, let $\etal_t(z)$ denote the set of particles present at site $z$ at time $t$, together with their states:
\[\etal_t(z)=\etal_t(z\pv\eta_0,\X,\Pb)=\{(x,i,\sigma)\in \Z^d\times\N\times\{1,\varrho\}\,:\,(Y^{x,i}_t,\gamma^{x,i}_t)=(z,\sigma)\}.\]
Thus, $\etal_t$ is a site-wise counterpart to $\zeta_t$, and knowing the former is equivalent to knowing the latter.
We write $\etalb_t=(\etal_t(z))_{z\in \Z^d}$ and $\etalb=(\etalb_t)_{t \geq 0}$.

Fix an increasing sequence $(U_n)_n$ of finite subsets of $\Z^d$ such that $U_n \uparrow \Z^d$.
For initial configurations $\eta_0$ with possibly infinitely many particles, we wish to define $\etalb(\eta_0,\X,\Pb)$ by: for all $T\ge0$, for all $z\in\Z^d$,
\begin{equation}\label{eq:definition_etal}
\etal_{|_{[0,T]}}(z\pv\eta_0,\X,\Pb)=\lim_n\ \etal_{|_{[0,T]}}(z\pv\eta_0\cdot\I_{w+U_n},\X,\Pb),
\end{equation}
where $w\in\Z^d$ would be arbitrary (and would thus give translation invariance), and the right-hand side sequence would be eventually constant in $n$. This leads us to

\begin{definition*}
Given $\xi=(\eta_0,\X,\Pb)$, the process $\etalb(\eta_0,\X,\Pb)$ is \emph{well-defined} if, for all $T\ge0$ and $z\in\Z^d$, for all $w\in\Z^d$, the sequence
\begin{equation}
	\etal_{|_{[0,T]}}(z\pv\eta_0\cdot\I_{w+U_n},\X,\Pb)\quad ;\quad n\ge0,
\nonumber
\end{equation}
is eventually constant and its limit does not depend on $w$. In this case, we define $\etalb(\eta_0,\X,\Pb)$ as in~\eqref{eq:definition_etal}.
\end{definition*}

Let us discuss a few important consequences of the above definition.

First, if $\etalb$ is well-defined, then the counting process $(\eta_t)_{t\geq 0}$, given
\[\eta_t(z\pv\eta_0,\X,\Pb)=\sum_{(x,i,\sigma)\in\etal_t(z)} \sigma,\]
has same law as in the site-wise definition.

If $\etalb$ is well-defined, then we may also define the processes $\zeta^{x,i}(\eta_0,\X,\Pb)$ for $x\in\Z^d$, $i\in\N$, by
\begin{equation}
	\zeta^{x,i}_{|_{[0,T]}}(\eta_0,\X,\Pb)=\lim_n\ \zeta^{x,i}_{|_{[0,T]}}(\eta_0\cdot\I_{w+U_n},\X,\Pb)
\end{equation}
where $w$ is arbitrary, and the right-hand side sequence is eventually constant. Indeed, for all $T>0$, for all $n$, $\zeta^{x,i}_{[0,T]}(\eta_0\cdot\I_{w+U_n},\X,\Pb)\subset X^{x,i}_{[0,T]}$, and $\etalb_{|_{[0,T]}}(\eta_0\cdot\I_{w+U_n},\X,\Pb)$ is eventually constant almost surely at all points of the finite set $X^{x,i}_{[0,T]}$, with a limit that does not depend on $w$.

By the definition, $\etalb$ is \emph{translation covariant}, i.e., if $\etalb(\xib)$ is well-defined, then $\etalb(\theta\xib)$ is well-defined and $\etalb(\theta\xib)=\theta\etalb(\xib)$, where $\theta\xib=(\theta\eta_0,\theta\X,\theta\Pb)$. As a consequence, the same holds for $\zetab$ as well.

Finally, if $\etalb$ is well-defined, then $\zetab$ is \emph{locally finite}, i.e., the number of particles $(x,i)$ visiting a given site $z$ before a given time $T$ is almost surely finite, since it is the limit of an eventually constant sequence of integers.

We are now in position to restate Theorem~\ref{thm:well_defined_intro} from the Introduction in a precise way:

\begin{theorem}\label{pro:well_defined}
Assume that $\displaystyle\sup_{x\in\Z^d}\E[\eta_0(x)]<\infty$. Then $\etalb(\eta_0,\X,\Pb)$ is a.s.\ well-defined.
\end{theorem}

We postpone the proof of Theorem~\ref{pro:well_defined} to Section~\ref{sec:proof_existence}. In the next section, we prove Proposition~\ref{pro:condition}, using the particle-wise construction and thus the above existence theorem.

\section{A sufficient condition for non-fixation}
\label{sec:nonfixation}

In this section we prove the non-fixation criterion, Proposition~\ref{pro:condition}.

Fixation as defined so far concerns the state of sites, and will be called here \emph{site fixation}.
In contrast, when each labeled particle eventually fixates, we call it \emph{particle fixation}. Due to Theorems~\ref{thm:agg} and~\ref{pro:well_defined}, it will suffice for us to prove that (\ref{eq:condition}) implies particle non-fixation.

The core of the proof therefore uses the particle-wise construction,
although a monotonicity property (Lemma~\ref{lem:monotone_particle} below) will also require a variant of the local site-wise construction.

\begin{proof}
[Proof of Proposition~\ref{pro:condition}]
We are assuming that $\eta_0(x)$, $x\in\Z^d$, are i.i.d.
We can moreover assume that $\E[\eta_0(o)]<\infty$, otherwise non-fixation holds by~\cite[Corollary 1.2]{AmirGurel-Gurevich10}.

In the particle-wise setting, ``stabilizing $V_n$'' means that we let the process evolve inside $V_n$ (with no particle outside), and stop particles at the boundary, until every particle either has become passive or has exited $V_n$ (since it is a finite-state Markov chain given $\eta_0$, the system reaches such an absorbing state in finite time). Recall also that we denote the law of this process by $\PP_{[V_n]}$.

Let us introduce
\[\Vt_n=V_{n-L_n},\]
where $L_n$ is an integer sequence (e.g.~$\lfloor\log n\rfloor$) such that
\[L_n \to \infty \qquad \text{ but } \qquad \frac{|V_n\setminus \Vt_n|}{|V_n|}\to 0.\]

For $n\in\NN$, introduce the event (which depends on the infinite labeled system)
\[A_n= \big\{\sup_t |Y^{o,1}_t| \ge L_n\big\}= \text{``particle $(o,1)$ reaches distance $L_n$''}.
\]

Let $\Mt_n$ be the number of particles originally in $\Vt_n$ that exit $V_n$. By definition of $\Vt_n$, and translation invariance of $\PP$,
\begin{align*}
\EE[\Mt_n]
	& = \sum_{x\in\Vt_n} \sum_{i\in\N} \PP(\text{particle $Y^{x,i}$ exits $V_n$})\\
	& \le \sum_{x\in\Vt_n} \sum_{i\in\N} \PP(\text{particle $Y^{x,i}$ reaches distance $L_n$ from $x$})\\
	& = |\Vt_n|\sum_{i\in\N}\PP(\text{particle $Y^{o,i}$ reaches distance $L_n$ from $o$})\\
	& = |\Vt_n|\sum_{i\in\N}\PP(\text{particle $Y^{o,1}$ reaches distance $L_n$ from $o$},\ \eta_0(o)\ge i)\\
	& \le |\Vt_n|\sum_{1\le i\le K}\PP(A_n)+|\Vt_n|\sum_{i> K}\PP(\eta_0(o)\ge i)\\
	& = |\Vt_n| K\, \PP(A_n)+|\Vt_n|\,\E[(\eta_0(o)-K)_+]
\end{align*}
for any $K\in\N$.
Therefore,
\[\EE[M_n] \le \EE[\Mt_n]+\mu |V_n\setminus\Vt_n| \le |\Vt_n|\big(K \PP(A_n)+\E[(\eta_0(o)-K)_+]\big)+o(|V_n|).\]
However, by Lemma~\ref{lem:monotone_particle} below,
\(\EE_{[V_n]}[M_n]\le \EE[M_n],\)
hence
\[\PP(Y^{o,1}\text{ does not fixate})=\lim_n \PP(A_n)\ge\frac1K\Big(\limsup_n \frac{\EE_{[V_n]}[M_n]}{|\Vt_n|}-\E[(\eta_0(o)-K)_+]\Big)>0\]
from the assumptions, provided $K$ is chosen large enough.

From Theorem~\ref{thm:agg}, we conclude that a.s.\ the system does not fixate, which finishes the proof of the proposition.
\end{proof}

It remains for us to consider the monotonicity property used above.

\begin{lemma}
\label{lem:monotone_particle}
Let $U$ be a finite subset of $\Z^d$. Let the random variable $M_U$ denote the number of particles starting in $U$ that visit $U^c$.
Then
\(
\E_{[U]}[M_U]\le \E[M_U].
\)
\end{lemma}

In order to show that the distribution of $M_U$ under $\PP_{[U]}$ is dominated by that under $\PP$, we use an intermediate step which is related to the particle-wise construction.
Namely, the distribution of $M_U$ under $\PP$ can be well approximated by a process that starts from $\eta_0 \cdot \I_{V}$ for big enough $V \subseteq \Z^d$. Let $\PP_{V}$ denote the law of such process.

Remark that, under $\PP_{[U]}$, that is, the dynamics where particles freeze outsize $U$, the unlabeled process is sufficient to determine the variable $M_U$, since the latter is simply the total number of particles found outside $U$ after stabilization.
Therefore, the distribution of $M_U$ under $\PP_{[U]}$ can be studied using either the particle-wise or site-wise construction.
However, this is not the case under $\PP$ or $\PP_{V}$ for any $V \supset U$. Indeed, in the unlabeled system, it is not possible to distinguish the particles that have exited and re-entered $U$ from the particles that have met them after their re-entrance.
In order to use the monotonicity properties in the spirit of the site-wise representation, the proof uses a two-color site-wise construction.
In this construction, particles which started in $U$ and have not yet exited $U$ are colored blue, and all other particles are colored red.
There are two stacks of instructions at each site, one for each color.
This way one can distinguish the particles which have not yet exited $U$ from those who have, and use this distinction to define $M_U$ without giving a different label to each particle.

\begin{proof}
[Proof of Lemma~\ref{lem:monotone_particle}]
Let us first show how the lemma follows from
\begin{equation}
\forall\ V \supset U \text{ finite, }\forall T>0,\qquad \PP_{[U]}(M_U^T\in\cdot)\lest \PP_{V}(M_U^T\in\cdot),\label{eq:mono_MT}
\end{equation}
where $M_U^T$ is the number of particles initially in $U$ that visit $U^c$ before time $T$.

First, for all $T>0$, it follows from the well-definedness of the particle-wise construction that the sequence of variables $M_U^T$ in a system which starts from $\eta_0\cdot\I_{U_n}$, for $n\in\N$, is a.s.\ eventually constant (cf.~Section~\ref{sec:particlewise}), for any fixed sequence $U_n \uparrow \Z^d$.
Hence, we can use~\eqref{eq:mono_MT} with $V=U_n$ to get
\[\PP_{[U]}(M_U^T\in\cdot) \lest \PP_{U_n}(M_U^T\in\cdot) \to \PP(M_U^T\in\cdot).\]
Taking expectation gives $\E_{[U]}[M_U^T]\le \E[M_U^T]$.
Finally, since $M^T_U\uparrow M_U$ as $T\to\infty$, we conclude that $\E_{[U]}[M_U]\le \E[M_U]$, finishing the proof.

We now prove~\eqref{eq:mono_MT} by a suitable coupling.
Let $V \supset U$ be finite and let $T>0$.

We introduce a two-colored particle system, where we are only interested at particle counts color by color: the configuration at time $t$ is $\eta(t)=(\eta^B(t),\eta^R(t))\in (\N_{0\varrho}\times\N_{0\varrho})^{\Z^d}$. Initially, blue~(B) particles are the particles that start from $U$, and red~(R) particles are the particles that start from outside $U$. Blue particles become red when they exit $U$, but otherwise, colorblind dynamics will be the same as for ARW.
The precise construction is as follows.

Let us consider two independent families $\boldsymbol\Ir^B$ and $\boldsymbol\Ir^R$ of Diaconis-Fulton instructions (see Section~\ref{sec:sitewise}), to be respectively used by blue particles  and red particles (thus we will only need $\Ir^B$ inside $U$). Let us also be given collections $\Pb^B$ and $\Pb^R$ of independent Poisson point processes at rate $1$, attached to each site, to be respectively used to trigger blue and red topplings. Similarly to the standard site-wise construction (Section~\ref{sec:sitewise}), this enables to construct the process $(\eta^B,\eta^R)$ from a finite initial configuration, with the difference that two clocks (B and R) now run at each site, at a speed given the respective numbers of B and R particles, and each clock triggers a toppling of the same color. Furthermore, blue topplings have the additional effect of changing the color of the jumping blue particle if it jumps out of $U$. Note that topplings affect the state of particles of both colors because red and blue at same site share activity: a red particle can prevent a blue particle from becoming passive, and vice-versa. This construction yields a natural coupling between the process restricted to $U$ (for which red particles play no role) and the process started from particles in $V$.

Unfortunately, this bi-color representation is not abelian.
Nevertheless, since each color uses a different stack of instructions, adding red particles has the only effect, regarding blue particles, of enforcing activation of some of them at some times. The part of the proof of Lemma~4 in~\cite[v2 on arXiv]{RollaSidoravicius12} that proves monotonicity of $L_t^M$ (and thus of $h_t^M$) in $M$ adapts here to show increase of the number of blue topplings $h_t^{(B)}$ when red particles are added: Property~2 is unchanged, and new activations only help said monotonicity.

Since red topplings do not change the total number $\eta^{(B)}(\Z^d)$ of blue particles, and blue topplings may only conserve or decrease that number, we deduce that $M_U^T=\eta^{(B)}_0(\Z^d)-\eta^{(B)}_T(\Z^d)$ cannot decrease when switching from the restricted process to the process started from $V$, which gives~\eqref{eq:mono_MT}.
\end{proof}

\section{Well-definedness of the particle-wise construction}
\label{sec:proof_existence}

In this section, we prove Theorem~\ref{pro:well_defined}. To that aim, let us first check that it suffices to prove the following:

\begin{proposition}\label{pro:limit_exists}
Assume the distribution $\eta_0$ is such that
\[\sup_{x\in \Z^d}\E[\eta_0(x)]<\infty.\]
Then, for any increasing sequence $(W_n)_n$ of finite subsets of $\Z^d$ such that $\bigcup_n W_n=\Z^d$, almost surely, for all $z\in \Z^d$ and all $T>0$, the sequence of processes
\begin{equation}
	\etal_{|_{[0,T]}}(z;\eta_0\cdot\I_{W_n},\X,\Pb);\quad n\ge0,
\nonumber
\end{equation}
is eventually constant.
\end{proposition}

\begin{proof}
[Proof of Theorem~\ref{pro:well_defined}.]
In order to prove well-definedness, it suffices to justify that the limit
\begin{equation}
	\etal_{|_{[0,T]}}(z\,;\,\eta_0,\X,\Pb)=\lim_n\ \etal_{|_{[0,T]}}(z\,;\,\eta_0\cdot\I_{w+U_n},\X,\Pb),
\end{equation}
which exists almost surely for all $w\in\Z^d$ due to the above proposition, does almost surely not depend on $w$.
This follows as well from the above proposition applied to an increasing sequence $(W_n)_n$ of subsets of $\Z^d$ that has subsequences in common with all sequences $(w+U_n)_n$ for $w\in\Z^d$.
\end{proof}

The proof of Proposition~\ref{pro:limit_exists} is based on the study of the propagation of the \emph{influence} of a particle, which we shall bound stochastically by a branching process. That way, we control the ``radius of influence'' of particles and discard the possibility that the state of a site depends on a infinite chain of influences, which would be the obstacle for existence.

\begin{definition*}
	For a finite configuration $\eta_0$, putative paths $\X$ and clocks $\Pb$, for $x,z\in \Z^d$, $i\in\N$ and $t\ge0$, the particle $(x,i)$ \emph{has an influence on site $z$ during $[0,t]$ in $(\eta_0,\X,\Pb)$}, and we shall write $(x,i)\!\!\influence{[0,t],\eta_0}\!\!z$, if $\eta_0(x)\ge i$ and
	\[\etal_{|_{[0,t]}}(z\,;\,\eta_0,\X,\Pb)\ne\etal_{|_{[0,t]}}(z\,;\,\eta_0-\delta_x,\X^*,\Pb^*),\]
	where $\X^*$ and $\Pb^*$ are obtained by removing the particle $(x,i)$ and shifting down the particles coming next, i.e., $X^{*(v,j)}=X^{(v,j+1)}$ if $v=x$ and $j\ge i$, and $X^{*(v,j)}=X^{(v,j)}$ else.
\end{definition*}

In other words, $(x,i)$ has an influence on $z$ during $[0,t]$ with initial condition $\eta_0$ if removing the particle $(x,i)$ in $\eta_0$ (but keeping the same putative walks and sleeping times everywhere else) would change the evolution of the process at $z$ at some time before $t$.

Assume that $\X$ and $\Pb$ are i.i.d.\ putative paths and clocks (as in the finite particle-wise construction).
For $x\in \Z^d$ and $t\ge0$, and any finite initial condition~$\pi$, we let
\[Z^{x,i}_t(\pi)=\{z\in \Z^d\,:\,(x,i)\!\!\influence{[0,t],\pi}\!\!z\}.\]
It is the set of vertices influenced during $[0,t]$ by the removal of the particle $(x,i)$ in $\pi$. Note that if $i>\pi(x)$ then $Z^{x,i}_t(\pi)=\emptyset$.

\begin{lemma}\label{lem:prop_Z}
One can construct a process $(\Zt_t)_{t\ge0}$ that takes values in finite subsets of $\Z^d$ and such that
\begin{itemize}
	\item for all $(x,i)\in \Z^d\times\N$, $t\ge0$, for any finite configuration $\pi$, $Z^{x,i}_t(\pi)\subset_{\rm st}x+\Zt_t$,\\
		i.e., for any $A\subset \Z^d$, $\PP(A\subset Z^{x,i}_t(\pi))\le \PP(A\subset x+\Zt_t)$;
	\item for all $t\ge0$, $\E\big[|\Zt_t|\big]\le e^{2(1+\lambda)t}$.
\end{itemize}
\end{lemma}

Let us deduce Proposition~\ref{pro:limit_exists} and then prove the lemma.

\begin{proof}
[Proof of Proposition~\ref{pro:limit_exists}.]
Assume $\sup_x\E[\eta_0(x)]<\infty$.

Let $z\in \Z^d$, $T>0$ and let $(W_n)_n$ be an increasing sequence of finite subsets of $\Z^d$ such that $\bigcup_n W_n=\Z^d$.
Note first that, up to introducing new terms inside the sequence, we may assume that $(W_n)_n$ is of the form
\[W_n=\{x_1,\ldots,x_n\},\]
where $\Z^d=\{x_n\,:\,n\in\N^*\}$ and the $x_n$'s are distinct.
We have
\begin{align}
	& \PP(\text{the sequence $(\etal_{|_{[0,T]}}(z\pv\eta_0\cdot\I_{W_n},\X,\Pb))_{n\ge0}$ is not eventually constant})\notag\\
	& =\PP\big(\text{for infinitely many $n$, $\etal_{|_{[0,T]}}(z\pv\eta_0\cdot\I_{W_n},\X,\Pb)\ne\etal_{|_{[0,T]}}(z\pv\eta_0\cdot\I_{W_{n-1}},\X,\Pb$})\big)\label{eq:major_stat}
\end{align}
However, for all $n$, since $W_n=W_{n-1}\cup\{x_{n}\}$, the fact that
\begin{equation}
	\etal_{|_{[0,T]}}(z\pv\eta_0\cdot\I_{W_n},\X,\Pb)\ne\etal_{|_{[0,T]}}(z\pv\eta_0\cdot\I_{W_{n-1}},\X,\Pb)\label{eq:ineq}
\end{equation}
implies that, for some $i\le\eta_0(x_n)$, $(x_n,i)$ has an influence on $z$ before time $T$ in $(\eta_0\cdot\I_{W_{n-1}}+i\delta_{x_n},\X,\Pb)$, i.e., $z\in Z^{x_n}_T(\eta_0\cdot\I_{W_{n-1}}+i\delta_{x_n})$. This is seen by adding particles at $x_n$ one by one until a change appears in the configuration at $z$ before time~$T$.

On the other hand,
\begin{align}
	&\E\Big[\#\{n\in \N\pp\exists i\le\eta_0(x_n),\,z\in Z_T^{x_n,i}(\eta_0\cdot\I_{W_{n-1}}+i\delta_{x_n})\Big]\notag\\
	& \le \sum_{n\in\N}\sum_{i\in\N}\PP(i\le\eta_0(x_n),\,z\in Z_T^{x_n,i}(\eta_0\cdot\I_{W_{n-1}}+i\delta_{x_n}))\notag\\
	& \le \sum_{n\in\N}\sum_{i\in\N}\PP(i\le\eta_0(x_n))\PP(z\in x_n+\Zt_T)\notag\\
	& = \sum_{x\in \Z^d}\E[\eta_0(x)]\PP(z\in x+\Zt_T)\notag\\
	& \le \sup_{x\in \Z^d}\E[\eta_0(x)]\sum_{x\in \Z^d}\PP(z\in x+\Zt_T)\notag\\
	& = \sup_{x\in \Z^d}\E[\eta_0(x)]\sum_{x\in \Z^d}\PP(2z-x\in z+\Zt_T)\label{eq:unimod}\\
	& =\sup_{x\in \Z^d}\E[\eta_0(x)]\E[|\Zt_T|]\notag\\
	& \le \sup_{x\in \Z^d} \E[\eta_0(x)]e^{2(1+\lambda)T}<\infty\notag,
\end{align}
where the (in)equalities follow from the properties listed in Lemma~\ref{lem:prop_Z}, and linearity of expectation. This proves that almost surely \eqref{eq:ineq} holds only for finitely many $n$, which shows that the last probability in~\eqref{eq:major_stat} equals zero, and concludes the proof of Proposition~\ref{pro:limit_exists}.
\end{proof}

\begin{proof}[Proof of Lemma~\ref{lem:prop_Z}]

Let $(x,i)\in \Z^d\times\N$, $T>0$ and a finite configuration $\pi$ be given. We aim at bounding the set of vertices influenced before time $T$ by the presence in $\pi$ of the particle $(x,i)$. We may assume $i\le\pi(x)$, for otherwise the bound is trivial.

Notice that particles behave independently \emph{except} when a particle gets reactivated by another --- we consider here that particles may always deactivate at rate~$\lambda$ but are reactivated instantaneously if another particle shares the same site ---, hence the influence only propagates in this case, and more precisely only when the reactivation may not be due to a third particle.
Thus, a particle is influenced by $(x,i)$ during $[0,t]$ if either it is the particle $(x,i)$ itself, or if it was reactivated from passive state before time $t$ by sharing the same site only with particles that were influenced before. If a particle is influenced, its inner time may change when the particle $(x,i)$ gets removed; in order to keep track of every possibility and bound the number of influenced particles, we shall consider all potential paths simultaneously, and have influenced particles both stay at a site and jump to the next one at the same time by doubling those particles.

Consider the system with initial condition $\pi$. In addition to the A (active) and S (passive, or sleeping) states, let us introduce new states I (influenced) and J (``trail'' of influenced particles), and modify the system as follows:
\begin{itemize}
	\item at time 0, the particle $(x,i)$ is put in state I (while the others in $\pi$ are A);
	\item particles in states A and S behave as in the usual ARW, i.e., jump at rate 1 (according to their putative path $\X$), deactivate at rate $\lambda$ (according to their clock $\Pb$) and reactivate as soon as another particle is present;
	\item particles in state I jump at rate 1 (according to their putative path given by $\X$), and leave a new particle in state J at the site they just left;
	\item particles in state J do not move (and thus do not need to be given a path or clock);
	\item when a particle in state S is reactivated by a particle in state I or J, and no particle of state A is present at the same site, then it switches from state S to state I.
\end{itemize}
It is important to note that the initial system is embedded in the new one, with some particles having state I or J instead of A and S. The state J enables to cover the case of particles slowed down by deactivations.
Also, it follows from the previous discussion that, at any time~$t$, all the vertices influenced by $(x,i)$ contain particles in state I or J. In other words, $Z^{x,i}_t$ is included in the set of locations of I and J particles at time~$t$.

The set $Z^{x,i}_t$ depends however on the restriction of the initial condition, i.e., on $\pi$. Still, one can \emph{stochastically} dominate $Z^{x,i}_t$, uniformly in $\pi$, by the process $x+\Zt_t$, where $\Zt_t$ is defined as the set of locations of particles in state I or J at time $t$ in the following simplified system:
\begin{itemize}
	\item at time 0, the system starts with only one particle, which is at $o$ and in state I;
	\item particles in state I jump at rate 1 (according to the jump distribution $p(\cdot)$), and leave a new particle in state~J at the site they just left and a new particle in state~I at the site they land on;
	\item particles in state J do not move;
	\item at rate $\lambda$, particles in state I or J produce a new particle in state I at the same site.
\end{itemize}
Indeed, the additional I particles are meant to cover the cases of conversion from A or S to I; they correspond to the case of fastest possible conversion and lead to an upper bound in general. In the case of jump of a particle in state I, the new particle introduced above covers the case when the target site contains already one particle in state S (there may be at most~1). In the last point, the new particle corresponds to a particle switching to state S that would be reactivated exclusively by I and J particles; since there may be at most one such particle, this happens at rate (at most) $\lambda$.

In summary, each particle gives birth to at most 2 new particles, at rate at most $\max(1,\lambda)\le 1+\lambda$, hence
\[\E[|\Zt_t|]\le e^{2(1+\lambda)t}.\]
This concludes the proof.
\end{proof}

\section{Well-definedness of the site-wise construction}
\label{sec:globalsitewise}

This section is not needed in the rest of the paper.
We show how the arguments from~\cite{RollaSidoravicius12} can actually be used to construct the whole evolution $(\eta_t)_{t\geq 0}$ as a translation-covariant function of $(\eta_0,\Ir,N)$.
Assume that
\begin{equation}
\label{eq:ergodicfinite}
\text{$\eta_0$ is spatially ergodic and $\E[\eta_0(o)]<\infty$.}
\end{equation}

In \S\ref{sec:sitewisefinite}, the site-wise representation $(\eta_0,\Ir)$ plus random clocks $N$ were used for a \emph{local site-wise construction} of the process.
The law $\PP_{[V]}$ indicated initial condition replaced by $-\infty$ outside $V$.
In \S\ref{sec:nonfixation} we introduced $\PP_{V}$ to indicate the initial condition being replaced by $0$ outside $V$.
This construction also has the property that occupation times $L^x(t)$, and thus the toppling counter $h_t(x)$, are increasing in the initial configuration $\eta_0$.
See the more detailed proof of~\cite[Lemma~4]{RollaSidoravicius12}, available at~\href{http://arXiv.org/abs/0908.1152v2}{arXiv:0908.1152v2}.

It follows from general results about particle systems on non-compact state spaces~\cite{Andjel82} that a stochastic process $(\eta_t)_{t\geq 0}$ corresponding to the ARW dynamics exists and is a Feller process with respect a certain topology, see~\cite[Footnote 6]{RollaSidoravicius12}.
This topology is stronger than product topology (and necessarily so, since it is always possible to find initial configurations that are empty in an arbitrarily large box and still can affect the configuration at $o$ at arbitrarily small times with arbitrarily high probability), but fortunately it is weak enough so that local events are described by a limit of finite systems.
More precisely, for every event $E$ that depends on a finite space-time window $\eta_{|_{[0,T]\times B(o,n)}}$ of the process $\eta$, under assumption~(\ref{eq:ergodicfinite}) we have
\begin{equation}
\label{eq:approx}
\PP(E) = \lim_{V\uparrow \Z^d} \PP_{V}(E).
\end{equation}

This local site-wise construction also has the property that, for each $x\in\Z^d$ and $T<\infty$, the expected number of times that other sites send particles to $x$ during $[0,T]$ is finite.

Monotonicity of $h_t$ in both $V$ and $t$ allow a certain limit swap. Using this and the last property above, one proves the equivalence between fixation for the evolution $(\eta_t)_{t\geq 0}$ and stabilizability of $\eta_0$ in the timeless site-wise representation.

All this is done without requiring that the process $(\eta_t)_{t\geq 0}$ be constructed directly from $(\eta_0,\Ir,N)$.
In fact, we are unaware of such \emph{global site-wise construction} of the ARW having been considered in the literature. Thanks to a question raised by the anonymous referee, when finishing the revision of the present paper we had the pleasant surprise to realize that the above properties are actually enough to ensure a global site-wise construction.

Assuming the reader went through the proof of~\cite[Lemma~4]{RollaSidoravicius12}, the field $h_t$ is increasing in $V$, so it has a limit that does not depend on the particular increasing sequence $V\uparrow \Z^d$.
Moreover, in the limiting field, for each fixed $T$ and each site $x$, $h_T(x)$ is finite and the set of sites $z$ such that $\Ir^{z,k}=x-z$ for some $k \leq h_T(z)$ is also finite.
So given $t\in[0,T]$ and a sequence $V\uparrow \Z^d$, $(h_t(z))$ is eventually constant for $z$ in this set, which by Local Abelianness implies that $\eta_t(x)$ is eventually constant, and does not depend on the particular sequence of domains $V$.
So convergence holds simultaneously for all $x\in\Z^d$ and $t\in\Q_+$, for $t\in\R_+$ we take limits from the right, and by~(\ref{eq:approx}) we can conclude that the resulting process $(\eta_t)_{t\geq 0}$ has the correct distribution.

\section{Comments on extensions of the results}
\label{sec:extensions}

Let us briefly comment on validity and relevance of our results in more general settings.

Let $G=(V,E)$ be a graph with finite degrees, and $p(\cdot,\cdot)$ be a transition kernel on $V$. Translation invariance is replaced by the following: Assume that there is a transitive unimodular subgroup $\Gamma$ of automorphisms that leave $p(\cdot,\cdot)$ invariant, i.e., $\Gamma$ is a subgroup of automorphisms and, if we write $S(x)=\{\gamma\in\Gamma\pp\gamma x=x\}$ for the stabilizer of $x\in V$,
\begin{align*}
	\text{(transitivity)}&\text{ for all $x,y\in V$, there is $\gamma\in\Gamma$ such that $\gamma x=y$;}\\
\text{(unimodularity)}&\text{ for all $x,y\in V$, we have $|S(x)y|=|S(y)x|$;}\\
	\text{(invariance)}&\text{ for all $x,y\in V$ and $\gamma\in\Gamma$, $p(\gamma x,\gamma y)=p(x,y)$}.
\end{align*}
Unimodularity implies mainly that the mass transport principle is satisfied: for any function $F:V\times V\to[0,1]$ that is $\Gamma$-invariant ($F(\gamma x,\gamma y)=F(x,y)$ for any $\gamma\in\Gamma$, $x,y\in V$),
\[\sum_{y\in V}F(x,y)=\sum_{y\in V}F(y,x)\]
We refer the reader to~\cite[Chapter 8]{LyonsPeres16} for a thorough introduction.

The above is the setting in which \cite{AmirGurel-Gurevich10} proves Theorem~\ref{thm:agg}.

The well-definedness of the particle-wise construction (Theorem~\ref{pro:well_defined}) carries on as well, under the same assumption $\sup_{x\in V}\E[\eta_0(x)]<\infty$, with the same proof, except that we choose an origin vertex $o\in V$
and in Lemma~\ref{lem:prop_Z}, $x+\Zt_t$ is replaced by $\gamma_x\Zt_t$ where, for all $x\in V$, $\gamma_x\in\Gamma$ is chosen arbitrarily so that $\gamma_x(o)=x$, and one can add the property that $\Zt_t$ is invariant in distribution under any $\gamma\in S(o)$.
Then the equality~\eqref{eq:unimod} is changed with the relation
\[\sum_{x\in V}\PP(z\in\gamma_x(\Zt_t))=\sum_{x\in V}\PP(x\in\gamma_z(\Zt_t)),\]
which is a consequence of the mass transport principle. Indeed, the function $F:(x,z)\mapsto\PP(x\in\gamma_z(\Zt_t))$ is $\Gamma$-invariant because of the previously mentioned invariance of $\Zt_t$.

In particular, these extensions of Theorems~\ref{pro:well_defined} and~\ref{thm:agg} together give that for i.i.d.\ initial conditions of mean $\mu>1$, the ARW a.s.\ does not fixate (cf.~\cite[Corollary 1.3]{AmirGurel-Gurevich10}).

It is then a simple matter to extend Proposition~\ref{pro:condition} to amenable graphs as follows:

\begin{proposition}
\label{pro:condition_extended}
Let $G=(V,E)$ be an amenable graph, and assume that the jump distribution $p(\cdot,\cdot)$ is invariant under a transitive group of automorphisms of $G$.
Let $V_n \subseteq V$ be finite and satisfy $\frac{|\partial V_n|}{|V_n|}\to 0$.
Consider the ARW model with i.i.d.\ integrable initial configuration on $G$. For each $n\in \N$, consider the law $\PP_n$ of its evolution restricted to $V_n$ and let $M_n$ be the random variable counting how many particles exit $V_n$ during stabilization.
If
\[
\limsup \frac{\E_n M_n}{|V_n|}>0,
\]
then the system does not fixate.
\end{proposition}

The only modification in the proof consists in redefining $\Vt_n\subset V_n$, which must satisfy $|V_n\setminus\Vt_n|=o(|V_n|)$.
Recall also that any transitive subgroup of automorphisms of an amenable graph is unimodular (Soardi and Woess, cf.~\cite{LyonsPeres16}).

As a motivation for such generality, one may notice that, with this proposition, the proof of Theorem~\ref{thm:main} seamlessly enables to obtain non-fixation for ARW on $\ZZ\times G$ where $G$ is an amenable, transitive graph, provided the jump distribution $p(\cdot,\cdot)$ is invariant under a transitive group of automorphisms and such that the random walk $(X_n)_n$ with jump distribution $p(\cdot,\cdot)$ on $\ZZ\times G$ satisfies almost surely $\pi_1(X_n)\limites{}n+\infty$,
where $\pi_1$ is the projection onto $\Z$, and the sleep rate $\lambda$ is small enough.

\section*{Acknowledgments}

This research started during thematic trimester ``Disordered Systems, Random Spatial Processes and Some Applications'' at the Henri Poincaré Institute. LR thanks the IHP and the Fondation Sciences Mathématiques de Paris for generous support. Both authors thank MathAmSud grant LSBS-2014 for supporting this collaboration in different occasions.
LT thanks the Department of Mathematics at the University of Buenos Aires for hospitality. This project was further supported by grants PICT-2015-3154, PICT-2013-2137, PICT-2012-2744, PIP 11220130100521CO, Conicet-45955 and MinCyT-BR-13/14.

\renewcommand{\baselinestretch}{1}
\parskip 0pt
\small
\bibliographystyle{bib/leo}
\bibliography{bib/leo}

\end{document}